\documentclass[a4paper,11pt,fleqn]{article}

\usepackage{fullpage}
\usepackage{amsmath}
\usepackage{amsthm}
\usepackage{amssymb}
\usepackage{enumerate}
\usepackage{hyperref}
\usepackage{xcolor}
\usepackage{graphicx}
\usepackage[export]{adjustbox}
\usepackage{dsfont} 

\theoremstyle{plain}

\newtheorem{theorem}{Theorem}[section]
\newtheorem{proposition}[theorem]{Proposition}

\newtheorem{corollary}[theorem]{Corollary}

\theoremstyle{remark}

\newtheorem{observation}[theorem]{Observation}
\newtheorem{remark}[theorem]{Remark}

\theoremstyle{definition}

\newcommand{\abs}[1]{\left\vert#1\right\vert}

\title{Concrete method for recovering the Euler characteristic of quantum graphs\footnote{Accepted by ``Journal of Physics A: Mathematical and Theoretical" (\href{https://doi.org/10.1088/1751-8121/ab95c1}{DOI: 10.1088/1751-8121/ab95c1}). Work under \href{https://creativecommons.org/licenses/by/4.0/deed.ast}{Creative Commons Attribution 4.0 licence} .}}
\author{Corentin L\'ena\footnote{\noindent Department of Mathematics, Stockholm University, 10691 Stockholm, Sweden. \newline Emails: \texttt{corentin@math.su.se}, \texttt{andrea.serio@math.su.se}.} \ and Andrea Serio\footnotemark[2]}
\date{4th June 2020}

\begin{document}
 \maketitle

\begin{abstract}
    Trace formulas play a central role in the study of spectral geometry and in particular of quantum graphs.
    The basis of our work is the result by Kurasov which links the Euler characteristic $\chi$ of metric graphs to the spectrum of their standard Laplacian.
    These ideas were shown to be applicable even in an experimental context where only a finite number of eigenvalues from a physical realization of quantum graph can be measured.
    
    In the present work we analyse sufficient hypotheses which guarantee the successful recovery of $\chi$.
    We also study how to improve the efficiency of the method and in particular how to minimise the number of eigenvalues required.
    Finally, we compare our findings with numerical examples---surprisingly, just a few dozens of eigenvalues can be enough.
\end{abstract}

\vspace{2pc}


\section{Introduction}

\subsection{Physical model}

Schr\"odinger operators on metric graphs have long been studied as simplified models of complex quantum mechanical systems, notably large molecules \cite{Pa36,Ru53}.
More recently, under the name of \emph{quantum graphs}, they have been widely used to describe nano-wires, wave-guides and networks \cite[Chapter 8]{Ex15}.
They have also been investigated by mathematicians as a simplified setting in which to understand complex problems in spectral theory, such as quantum chaos \cite{KoSm99}, trace formulas \cite{Ro84,KoSm99,Gu01,KuNo05,KuNo06} and nodal patterns \cite{Ba12}.
From this point of view, they stand midway between one-dimensional and multi-dimensional Schr\"odinger operators.
Indeed, metric graphs are one-dimensional objects with a non-trivial geometry.
The central theme in studying quantum graphs is therefore the interaction of the spectrum with the geometry of the underlying metric graph. 

\subsection{Mathematical problem}

We consider in this work a finite, compact and connected metric graph $\Gamma$. Following \cite{Ku20}, we describe it by a set of $N$ edges $\{e_1,\dots,e_N\}$, with $e_n=[x_{2n-1},x_{2n}]$, and a set of $M$ vertices $\{v_1,\dots,v_M\}$, which is a partition of the set of endpoints $\{x_p\,:\,1\le p\le 2N\}$. The associated quantum graph is an operator in the Hilbert space 
\begin{equation*}
    L_2(\Gamma)=\oplus_{n=1}^N L_2(e_n).
\end{equation*}
More precisely, we define the \emph{standard Laplacian} $L_\Gamma$ on $\Gamma$, which is the object of our study, in the following way. Its domain $\mbox{Dom}(L_\Gamma)$ consists of functions $f\in L_2(\Gamma)$ satisfying
\begin{enumerate}
    \item for every $1\le n\le N$, $ f|_{e_n}\in W_2^2(e_n)$;
    \item for every $1\le m\le M$,
    \begin{enumerate}
        \item $f(x_p)=f(x_q)$ whenever the endpoints $x_p$ and $x_q$ belong to the vertex $v_m$ (continuity condition),
        \item $\sum_{x_p\in v_m}\partial f(x_p)=0$ (balance condition).
    \end{enumerate}
\end{enumerate}
In the balance condition, $\partial f(x_p)$ is the derivative along the edge with endpoint $x_p$, taken towards the inside of the edge. The conditions in (ii) are called the standard vertex conditions. The action of $L_\Gamma$ on $f\in \mbox{Dom}(L_\Gamma)$ is then given by the differential expression $L_\Gamma f(x)=-f''(x)$ for all $n\in \{1,\dots,N\}$ and $x$ in the interior of $e_n$. 

It is known that the operator $L_\Gamma$ is self-adjoint, non-negative and has compact resolvent.
Its spectrum therefore consists of a sequence of non-negative eigenvalues $(\lambda_j)_{j\ge1}$ of finite multiplicity tending to $+\infty$ (we repeat them according to their multiplicities).
Let us note that $\lambda_1$ is always $0$ and that the hypothesis that $\Gamma$ be connected implies that it is simple, with the associated eigenfunction proportional to the constant function $1$.
For the sake of convenience, we will work in the rest of the paper with the \emph{eigenfrequencies} $k_j=\sqrt{\lambda_j}$.   

Our goal is to recover information about the geometry of the graph from the knowledge of some of its eigenfrequencies.
This is a variation on Kac's famous problem ``Can one hear the shape of a drum?" \cite{Ka66}, with some specificities that we will address below. The analogue to Kac's problem in our setting would  be showing that two graphs with the same spectrum are isometric.
There are partial results in that direction. For instance, the eigenfrequencies of a quantum graph satisfy Weyl's law: as $j\to+\infty$, $k_j\sim\pi j/{\mathcal{L}}$, where $\mathcal L$ is the total length of the graph, i.e. the sum of the lengths of all its edges ({see for instance \cite{Be13}}). The total length of the graph can therefore be read off the complete spectrum. Furthermore, the answer to Kac's problem is positive provided all edge-lengths of the graph are rationally independent \cite{Gu01, KuNo05}. Nevertheless, if this last condition is not satisfied, we can find graphs with the same spectrum that are not isometric \cite{Gu01, Ba06, Ba09, Ba10}.

We will restrict ourselves to recovering the main topological invariant of a connected graph: its Euler characteristic $\chi$, defined by $\chi=M-N$. It is related to $\beta_1$, the number of independent cycles in the graph, by
$\beta_1=1-\chi$. { Let us point out that $\beta_1$ can also be recovered from the number of zeros of the eigenfunctions, as shown in \cite{AlBaBe18}. These numbers are however less easily accessible to experiments.}

\subsection{Objectives}

We focus on a situation that can be realized experimentally.
A microwave network can be described to a good approximation by a quantum graph, whose eigenfrequencies correspond to resonances of the system.
In \cite{Mi20} the authors describe the measurement of about 50 resonances of such networks and use this experimental result to recover $\chi$.
Contrary to expectations, they find this relatively small number of eigenfrequencies to be sufficient.
Their computation is based on a trace formula as we will see below in more details.
Extending the work in this paper, we want to check the correctness of the method and modify it in order to improve the efficiency, or namely to reduce the number of eigenfrequencies required.

The total length of the graph is a continuous parameter and an infinite number of eigenfrequencies are needed to compute it using Weyl's law, as described above.
By contrast, $\chi$ is integer valued and we can hope to compute it with only a finite number of them.
However, some assumptions on the graph are needed.
Indeed, by attaching a very small loop to some edge in an existing graph, we decrease $\chi$ by $1$ while altering the first $50$ eigenfrequencies by an arbitrarily small amount.
We come back to this point at the end of Section \ref{se:numerics}.
Therefore, our goal will also be to determine a reasonable set of assumptions that guarantees the success of our methods.

\section{Trace formula}
\label{sec:Trace}

\subsection{Main tool}

{All our methods for computing the Euler characteristic are based on the trace formula in \cite{KuNo05, KuNo06} (building on previous work in \cite{Ro84,KoSm99, Gu01}).}
This formula connects the spectrum of $L_\Gamma$, defined in the previous section, with the set of periodic orbits in $\Gamma$, to be described below. More precisely, the spectrum is described by the distribution 
\begin{equation}
\label{eq:def_u}
    u(k):=2\delta(k)+\sum_{k_j>0}\Big(\delta(k-k_j)+\delta(k+k_j)\Big),
\end{equation}
which is a tempered distribution of order $0$.
We denote by $\mathcal P$ the set of closed continuous paths in $\Gamma$ that are allowed to turn back on themselves only at vertices, which we call \emph{periodic orbits}, or \emph{orbits}.
If $p\in \mathcal P$, we denote by $\mbox{prim}(p)$ any shortest orbit such that $p$ can be obtained by repeating $\mbox{prim}(p)$.
Thus any two such orbits have the same length, hence $\ell(\mbox{prim}(p))$ is well defined.
The following quantity will play a crucial role in our methods:
\begin{equation}\label{def:lmin}
    \ell_{\mbox{\tiny min}}:=\min\{\ell(p)\,:\,p\in\mathcal P\}.
\end{equation}{
\begin{observation}
    Notice that $\ell_{\mbox{\tiny min}}$ is given by the minimum between the length of the shortest loop and twice the length of the shortest edge.
\end{observation}}

\begin{proposition}[Theorem 1 from \cite{KuNo05, KuNo06}, also \cite{Ku20}]\label{prp:trace_formula}
    We have
\begin{align}
    \label{eq:u}
    u(k) &= \chi \delta(k) + \frac{\mathcal{L}}{\pi} + \frac{1}{\pi} \sum_{p \in \mathcal{P}} \ell(\mbox{prim}(p))S_V(p) \cos k \ell(p),
 \end{align}
where $S_V(p)$ is the product of all the vertex scattering coefficients along the path $p$. { The convergence of the series occurs in the space $\mathcal S'(\mathbb R)$ of tempered distributions.}
\end{proposition}

{
\begin{remark}
The terms $S_V(p)$ will ultimately drop out of our formulas due to our choice of test functions. Since they will not play a role in our analysis we just refer to \cite{KuNo05} for their definition.
\end{remark}}

\begin{remark}
Taking the Fourier transform on both sides of Equation (\ref{eq:u}), we obtain
\begin{equation}
\label{eq:hat_u}
	 \hat u(\ell)=\chi+2\mathcal L \delta(\ell)+\sum_{p\in\mathcal P}\ell(\mbox{prim}(p))S_V(p)\Big(\delta(\ell-\ell(p))+\delta(\ell+\ell(p))\Big),
\end{equation}
which is the form of the Trace Formula that we will be using in the rest of the paper.
\end{remark}

{In all the paper, we define the Fourier transform so that $\hat f(k)=\int_{-\infty}^{+\infty} e^{ik\ell}f(\ell)\,d\ell$ for any $f\in L_1(\mathbb R)$.}
By definition of the Fourier transform of a tempered distribution, for $\varphi$ in the Schwartz class $\mathcal S(\mathbb{R})$, $\hat u[\varphi]=u[\hat\varphi]$, which means, using Equations (\ref{eq:def_u}) and (\ref{eq:hat_u}),
\begin{align}
 \chi\int_{-\infty}^{+\infty}\varphi(\ell)\,d\ell+2\mathcal L\varphi(0)+\sum_{p\in\mathcal P}\ell(\mbox{prim}(p))S_V(p)\left(\varphi(-\ell(p))+\varphi(\ell(p))\right)\nonumber\\
	=2 \hat {\varphi}(0) +\sum_{k_j>0}\Big(\hat {\varphi}(k_j)+\hat {\varphi}(-k_j)\Big).\label{eq:series}
\end{align}
Equation (\ref{eq:series}) also holds for some functions not in $\mathcal S(\mathbb R)$.
For our purpose, we will consider the set $\mathcal T$ of functions $\varphi$ satisfying the following assumptions:
\begin{enumerate}
    \item $\varphi$ is continuous on $\mathbb R$ and nonnegative;
	\item $\mbox{supp}(\varphi) \subseteq [0,T]$ for some $T > 0$;
	\item $\int_{-\infty}^{+\infty} \varphi(\ell)\,d\ell= 1$;
	\item the function $F_\varphi$ is integrable on $(0,+\infty)$, where $F_\varphi(k):=\sup_{y\ge k}|\Re(\hat \varphi(y))|$.
\end{enumerate}
Let us make some comments on Condition (iv).
The function $F_\varphi$ can be seen as the smallest non-increasing majorant of $|\Re(\hat\varphi)|$ on $(0,+\infty)$.
Its integrability is used both to express $\chi$ as the sum of a convergent series and to prove error estimates.

\begin{proposition}
    If $\varphi\in \mathcal T$, 
\begin{equation}
    \label{eq:series2}
 \chi+\sum_{p\in\mathcal P}\ell(\mbox{prim}(p))S_V(p)\Big(\varphi(-\ell(p))+\varphi(\ell(p))\Big)=
	2 \hat {\varphi}(0) +2\sum_{k_j>0}\Re\left(\hat {\varphi}(k_j)\right).
\end{equation}
\end{proposition}
We prove the above proposition in \ref{appTrace}.
Let us note that whenever $\varphi$ is in $\mathcal T$, so is $\varphi_t(l):=t\varphi(t\ell)$, for any $t>0$. 

We can now state the result by which we recover the Euler characteristic of $\Gamma$.

\begin{proposition}\label{prp:exact}
    For any $\varphi \in \mathcal{T}$ such that $T=1$ in (ii) then
    \begin{equation}
        \label{eq:chi_series}
        \chi= 2 \hat {\varphi_t}(0) +2\sum_{k_j>0}\Re\left(\hat {\varphi_t}(k_j)\right) \qquad \forall t \geq 1/\ell_{\mbox{\tiny min}},
    \end{equation}
    where $\ell_{\mbox{\tiny min}}$ is defined in (\ref{def:lmin}).
\end{proposition}
\begin{proof}
Since $t \geq 1/\ell_{\mbox{\tiny min}}$, the support of $\varphi_t$ is contained in $[0,\ell_{\mbox{\tiny min}}]$.
Equation (\ref{eq:chi_series}) then follows immediately from Equation (\ref{eq:series2}). 
\end{proof}
In the rest of the paper, we assume that $t \ge 1/\ell_{\mbox{\tiny min}}$.
When we will analyse our method, it will accordingly be important to know a lower bound of $\ell_{\mbox{\tiny min}}$ in order to guarantee a successful estimate of $\chi$. 
{    
On the other hand, from the above proposition we also have the following corollary, cf. \cite{Ku08}.
\begin{corollary}\label{cor:exact}
    For any $\varphi \in \mathcal{T}$ then
    \begin{equation}
        \label{eq:chi_limit}
        \chi=  \lim_{t \to \infty} \left( 2 \hat {\varphi_t}(0) + 2\sum_{k_j>0}\Re\left(\hat {\varphi_t}(k_j)\right) \right).
    \end{equation}
\end{corollary}
    The limit and summation signs in (\ref{eq:chi_limit}) cannot be exchanged since $\lim_{t \to \infty} \Re\left(\hat {\varphi_t}(k)\right) = \hat {\varphi}(0)$ independently of $k$.
    Hence it is clear that (\ref{eq:chi_limit}) cannot be approximated by the limit of a finite sum over the  smallest, say $J$, eigenfrequencies.
    Because of Assumption (iv) on $\varphi$, the summation will be shown to be absolutely converging (see Section \ref{sec:gen}).
    Based on that, our plan is to quantify for each $t$ how many eigenfrequencies $J$ are necessary for the truncated summation to differ from (\ref{eq:chi_limit}) by less than $1/2$, and hence the nearest integer to the sum to coincide with $\chi$.
    Theorem \ref{thm:opt_d} gives an optimal (minimal) number $J^*$  of required eigenfrequencies, based on the choice of $t$.
    As one can expect the worse the estimate of $\ell_{\mbox{\tiny min}}$ (smaller lower bound), hence the larger $t$ is, the larger the number $J^*$ of required eigenfrequencies.
}

\begin{remark}
    {The above trace formula holds for the standard Laplacian, specifically for standard vertex conditions imposed at all the vertices.
    In \cite{Bo09, Ku10} are presented versions of the trace formula for more general vertex conditions  where in particular the constant term in $\hat{u}$ may differ from $\chi$.}
    Hence our method might not be applicable in presence of vertex conditions other than standard.
    It is even possible that $\chi$ cannot be recovered knowing the whole spectrum of the graph: see \cite{Ba06, Ba09, Ba10} for an example of two graphs with mixed vertex conditions which have distinct values of $\chi$ despite being isospectral.
\end{remark}

\subsection{Choice of the test function}

Let us apply Formula (\ref{eq:chi_series}) to a triangular function: 
\begin{equation*}
    \psi(\ell) := 4\ell \mathds{1}_{[0,1 / 2]}(\ell) + 4(1-\ell)\mathds{1}_{[1 / 2, 1]}(\ell).
\end{equation*}
Everywhere in the paper, $\mathds{1}_A$ stands for the characteristic function of the set $A$. We find
\begin{align*}
    \label{eq:FT_tr}
    &\hat\psi(k) := \int_0^1\psi(\ell)e^{ik\ell}\,d\ell
	=e^{ik/2} \left(\frac{\sin(k/4)}{k/4}\right)^2.
 \end{align*}
Proposition (\ref{prp:exact}) then gives us the following expression for $\chi$, assuming $t \geq 1 / \ell_{\mbox{\tiny min}}$:
\begin{equation}
        \label{eq:chi_tr}
        \chi= 2+2\sum_{k_j>0}\cos(k_j/2t)\left(\frac{\sin(k_j/4t)}{k_j/4t}\right)^2.
\end{equation}

We can also apply the same formula to a truncated trigonometric function:
\begin{equation*}
    \varphi(\ell) := (1-\cos(2 \pi \ell))\mathds{1}_{[0,1]}(\ell)
\end{equation*}
We find
\begin{align*}
    \label{eq:FT_cos}
    \hat\varphi(k) &:= \int_0^1\left(1-\frac{e^{2i\pi\ell}}{2}-\frac{e^{-2i\pi\ell}}{2}\right)e^{ik\ell}\,d\ell\\
	&=-\frac{8\pi^2}{k(k-2\pi)(k+2\pi)}e^{ik/2}\sin(k/2)\\\
 \end{align*}
We deduce another expression for $\chi$, valid whenever $t \geq 1 / \ell_{\mbox{\tiny min}}$:
\begin{equation}
        \label{eq:chi_cos}
        \chi= 2-8\pi^2\sum_{k_j>0}\frac{\sin(k_j/t)}{k_j/t(k_j/t-2\pi)(k_j/t+2\pi)}.
\end{equation}
This formula is the one employed in \cite{Mi20}.

In order to recover $\chi$ from a finite set of eigenfrequencies, we have to replace the series on the right-hand side of Formulas (\ref{eq:chi_tr}) and (\ref{eq:chi_cos}) by partial sums.
Since the terms in the second series go to zero more quickly than in the first, we can expect that fewer eigenfrequencies are necessary to find a close approximation of $\chi$, which we will then be rounded (up or down) to the nearest integer.
Both numerical trials and the analysis of the experimental results in \cite{Mi20} confirm this.
The more regular $\varphi$ is, the faster $\hat{\varphi}$ decays at infinity.
We can then wonder how much we can decrease the number of eigenfrequencies used to compute $\chi$.
To study this question, we will consider the one-parameter family of test functions
\begin{equation*}
    \varphi_d(\ell)=C_d\left(1-\cos(2\pi\ell)\right)^d\mathds{1}_{[0,1]}(\ell),
\end{equation*}
where $d$ is a positive integer{, which we call \emph{order},} and $C_d$ is chosen so that $\int_{-\infty}^{+\infty}\varphi_d(\ell)\,d\ell=1$.
We easily check that the function $\varphi_d$ is of class $\mathcal{C}^{2d-1}$ on $\mathbb R$. { We show in \ref{appFT} that 
\begin{equation*}
    C_d=\frac{2^d(d!)^2}{(2d)!}
\end{equation*}
and that the Fourier transform of $\varphi_d$ is 
\begin{equation*}
    \hat \varphi_d(k)=\frac{(-1)^d(d!)^2}{\pi\Pi_{j=-d}^d(k/(2\pi)+j)}e^{ik/2}\sin(k/2).
\end{equation*}
}
Our goal is to give rigorous error estimates and to perform numerical trials for this family of functions.

\section{General estimate}
\label{sec:gen}
\subsection{Calculation with exact spectrum}

For a fixed graph $\Gamma$, we assume that only the smallest $J$ eigenvalues of the standard Laplacian are known. We consider the distribution (\ref{eq:def_u}) evaluated on the Fourier transform of a generic function $\varphi \in \mathcal{T}$ and split the expression in the following way
\begin{equation}\label{eq:SR_def}
        u[\hat{\varphi}_t] = \underbrace{2 \sum_{j \leq J} \Re\left(\hat {\varphi_t}(k_j)\right)}_{\textstyle S_J(t)} + \underbrace{2\sum_{j > J} \Re\left(\hat {\varphi_t}(k_j)\right)}_{\textstyle 
         R_J(t)}.
\end{equation}
We call $S_J(t)$ the \emph{truncated sum} and $R_J(t)$ the \emph{remainder} or \emph{tail}.
We have
\begin{equation}\label{eq:triang_ineq}
    |S_J(t) - \chi| \leq |u[\hat{\varphi}_t] - \chi| + |R_J(t)|.
\end{equation}
Let $\varphi \in \mathcal{T}$ such that $T = 1$ in Condition (ii) be fixed. Because of Proposition \ref{prp:exact} when $t \geq 1/\ell_{\mbox{\tiny min}}$ the term $|u[\hat{\varphi}_t] - \chi|$ is zero, hence $|S_J(t) - \chi| \leq |R_J(t)|$.
\begin{align}
        |R_J(t)|    &\leq 2\sum_{j > J} |\Re \hat{\varphi}(k_j / t)| \\
                    &\leq 2\sum_{j > J} \sup_{k \geq k_j } \abs{\Re \hat{\varphi}(k / t)} = 2\sum_{j > J} F_\varphi (k_j / t),
\end{align}
where $F_\varphi$, is defined in Condition (iv) of $\mathcal{T}$.
In particular, $F_\varphi$ is positive decreasing to zero \footnote{This follows immediately from the fact that $\Re \hat{\varphi}$ has only isolated zeros. Indeed, it is a real-analytic function which is not identically zero.}.
At this point we employ the classical eigenvalue estimate $k_j \geq (j - M) \pi  / \mathcal{L}$ (see for instance \cite{Be17}) and we bound the summation over the integers of a monotone function by its integral.
\begin{align}
        |R_J(t)|    &\leq 2\sum_{j > J - M} F_\varphi \left( \frac{\pi}{\mathcal{L}}\frac{j}{t}  \right)\\
                    &\leq 2 \int_{J - M}^\infty F_\varphi \left( \frac{\pi}{\mathcal{L}}\frac{x}{t} \right) \,dx \\
                    &= 2\frac{\mathcal{L} t}{\pi} \int_{\frac{\pi}{\mathcal{L} t}(J - M)}^\infty F_\varphi \left( \mbox{x} \right) \,d\mbox{x} =: \mathcal{E}_\varphi(J - M,\mathcal{L}t).
\end{align}
The function $\mathcal{E}_\varphi:\mathbb{R}^2_+ \to \mathbb{R}_+$ is continuous, positive, and monotone with respect to both arguments.
In particular $\mathcal{E}_\varphi(J-M, \cdot)$ is strictly monotone increasing, while $\mathcal{E}_\varphi(\cdot,{  \mathcal{L}t})$ is strictly monotone decreasing to zero.
Since $\chi$ is an integer number (generally negative) it is enough to ensure that $|R_J(t)| < 1/2$ in order to compute $\chi$.
Therefore we define $J_{\mbox{\tiny min}}(M,\mathcal{L}t) :=\min \{ J \in \mathbb{N} : \mathcal{E}_\varphi(J - M,\mathcal{L}t) < 1/2 \}$, this is the minimum number of eigenvalues needed for computing $\chi$.
This can be further extended to the case where only partial information on $\Gamma$ is known. The result is summarised in
\begin{theorem}\label{thm:estimate_exact}
Let $\Gamma$ be a graph which
\begin{itemize}
    \item has at most $\overline{M}$ vertices, i.e. $M \leq \overline{M}$;
    \item has total length at most $\overline{\mathcal{L}}$, i.e. $\mathcal{L} \leq \overline{\mathcal{L}}$;
    \item has its shortest orbits at least $\underline{\ell_{\mbox{\tiny min}}}$ long, i.e. $\ell_{\mbox{\tiny min}} \geq \underline{\ell_{\mbox{\tiny min}}}$.
\end{itemize}
Let $\varphi \in \mathcal{T}$ with $T=1$ and let $t = 1 / \underline{\ell_{\mbox{\tiny min}}}$.
Then $J \geq J_{\mbox{\tiny min}} = J_{\mbox{\tiny min}}(\overline{M},\overline{\mathcal{L}}t)$ eigenvalues of the standard Laplacian on $\Gamma$ suffice for recovering the Euler characteristic of $\Gamma$ via the formula
\begin{equation} \label{eq:exact}
    \chi(\Gamma) = \mbox{\emph{nint}} \left[ S_{J} \left( t \right) \right], 
\end{equation}
where $ \mbox{\emph{nint}}(x)$ denotes the nearest integer function.
\end{theorem} 
\begin{proof}
    Because $t = 1 / \underline{\ell_{\mbox{\tiny min}}} \geq 1 / \ell_{\mbox{\tiny min}}$ then Proposition \ref{prp:exact} implies $|\chi - u(\hat{\varphi_t})| = 0$, hence (\ref{eq:triang_ineq}) reads $|\chi - S_J(t)| \leq |R_J(t)| \leq \mathcal{E}_\varphi(J - M,\mathcal{L}t)$.
    Because of the monotonicity of $\mathcal{E}_\varphi$, we have $\mathcal{E}_\varphi(J - M,\mathcal{L}t) \leq \mathcal{E}_\varphi(J - \overline{M},\overline{\mathcal{L}}t)$. Hence for $J \geq J_{\mbox{\tiny min}}$ it follows that $|\chi - S_J(t)| < 1/2$, and thus we obtain (\ref{eq:exact}).
\end{proof}
Notice that the cases $S_J(t) = \chi \pm 1/2$ is excluded, hence we do not need to worry of the definition of the function $\mbox{nint}[\cdot]$ over the half integers.

We notice that the hypothesis of the above theorem cannot be relaxed too much.
In particular, the condition $\ell_{\mbox{\tiny min}} \geq \underline{\ell_{\mbox{\tiny min}}}$ is crucial, because adding small loops to any graph perturbs its smallest eigenvalues only a little, but at the same time it does dramatically affect its topological structure, decreasing $\chi$ by $1$.

We also notice that the quantity $\overline{\mathcal{L}}$ can be obtained from other information on the graph when available.
For instance if the number of edges of the graph is at most $\overline{N}$ then $\displaystyle \mathcal{L} \leq \min_{2 \leq j \leq J} ({j} + \overline{N}) \pi / k_j$.

\begin{remark}\label{rmk:extension_thm_exact}
The above Theorem \ref{thm:estimate_exact} can be reproduced after replacing $F_\varphi$ by any other majorant function monotone decreasing. The $J_{\mbox{\tiny min}}$ computed will be different. We will use this remark in Section 4.
\end{remark}

\subsection{Calculation with approximate spectrum}
\label{ss:approxspectrum}
Since we have in mind to apply our method in an experimental context we need to consider the case where the eigenvalues are known only up to some approximation. We denote by $\{\widetilde{k}_j \}_{j=1}^J$ the approximate spectrum where each eigenfrequency is measured with precision $\delta$, meaning that $\abs{\widetilde{k}_j-k_j} \leq \delta$ for all $1 \leq j \leq J$.
Hence we define $\widetilde{S}_J(t) = 2 \sum_{j \leq J} \Re\hat {\varphi_t}(\tilde{k}_j)$ the truncated sum evaluated at the approximated eigenvalues, hence our new aim is to approximate $\chi$ from $\widetilde{S}_J(t)$.
Therefore to the error contribution $\mathcal{E}_\varphi(J - M,\mathcal{L}t)$ we need to add a new term
\begin{equation}
    \abs{\widetilde{S}_J(t) - \chi}  \leq \mathcal{E}_\varphi(J - M,\mathcal{L}t) + \abs{\widetilde{S}_J(t) - S_J(t)}.
\end{equation}

This second term can be bounded linearly in $\delta$ in the following way
{
\begin{align}
     \abs{\widetilde{S}_J(t) - S_J(t)}
    &= \abs{2\sum_{j \leq J} \Re (\hat {\varphi}(\widetilde{k_j}/t) - \hat {\varphi}(k_j/t))} \label{eq:error:start}\\
    &\leq 2 \abs{\sum_{j \leq J} \int_0^1  \left( \cos{\widetilde{k_j}\ell /t}  -  \cos{k_j \ell /t} \right)\varphi(\ell)\,d\ell} \\
    &\leq 2  \abs{ \sum_{j \leq J} \int_0^1 -2 \sin{\frac{ (\widetilde{k_j} - k_j) \ell}{2t}} \sin{\frac{(\widetilde{k_j} + k_j)\ell}{2t}}\varphi(\ell)  \,d\ell} \\
    &\leq  \frac{2\delta}{t} \sum_{j \leq J} \abs{ \int_0^1 \sin{\frac{(\widetilde{k_j} + k_j)\ell}{2t}}\varphi(\ell)  \,d\ell} \label{eq:error:step}\\
    &\leq  \frac{2 \delta J}{t}.\label{eq:error:end}
\end{align}
}
Therefore we define
\begin{equation}\label{eq:error_approx}
    \widetilde{\mathcal{E}}_\varphi(J - M,\mathcal{L},t) := \mathcal{E}_\varphi(J - M,\mathcal{L}t) + \frac{2\delta J}{t}.
\end{equation}
For this new function there does not necessarily always exists $J$ such that $\widetilde{\mathcal{E}}_\varphi(J - M,\mathcal{L},t) <1/2$.
For simplicity we consider separately the following two inequalities, which together are sufficient:
\begin{itemize}
        \item $\mathcal{E}_\varphi(J - M, \mathcal{L} t) < 1/4$;
        \item $2\delta J / t < 1/4$.
\end{itemize}
We derive from them the following two conditions:
\begin{itemize}
        \item \hfill 
        \makebox[0pt][r]{%
            \begin{minipage}[b]{\textwidth}
              \begin{equation}
                 \label{eq:cond1}
            J \geq \widetilde{J}_{\mbox{\tiny min}}(M,\mathcal{L}t) := \min \{ J \in \mathbb{N} : J > M,\, \mathcal{E}_\varphi(J - M,\mathcal{L}t) < 1/4 \};
              \end{equation}
          \end{minipage}}
        \item \hfill 
        \makebox[0pt][r]{%
            \begin{minipage}[b]{\textwidth}
              \begin{equation}
                 \label{eq:cond2}
            \delta \leq \widetilde{\delta}_{\mbox{\tiny max}}(t,J) := t / (8J).
              \end{equation}
          \end{minipage}}
\end{itemize}
If the first $J$ eigenvalues are known such that both the above conditions are satisfied then we are able to compute the Euler characteristic via the approximated spectrum.
The above analysis can be summarised in the following

\begin{theorem}
    \label{thm:estimate_approx}
Let $\Gamma$, $\varphi$ and $t$ be as in the hypotheses of Theorem \ref{thm:estimate_exact}, and moreover let the first $J$ eigenvalues of the standard Laplacian on $\Gamma$ be known up to error $\delta$. Assume that $J$ and $\delta$ satisfy both conditions (\ref{eq:cond1}) and (\ref{eq:cond2}) with $M, \mathcal{L}$ replaced by $\overline{M}, \overline{\mathcal{L}}$. Then the Euler characteristic of $\Gamma$ can be recovered via the following formula
\begin{equation} \label{eq:approx}
    \chi(\Gamma) = \mbox{\emph{nint}} \left[ \widetilde{S}_{J} \left( t \right) \right]. 
\end{equation}
\end{theorem}

\section{Example of a family of test functions}

\subsection{Error estimate and optimal order}

We now apply the analysis in Section \ref{sec:gen} to { the test functions $\varphi_d$ introduced at the end of Section \ref{sec:Trace}}. In the rest of the section, we assume that $t=1/\underline{\ell_{\mbox{\tiny min}}}$. We have
\begin{equation*}
    \Re\left(\hat \varphi_d(k)\right)=\frac{(-1)^d(d!)^2\sin(k)}{2\pi\Pi_{j=-d}^d(k/(2\pi)+j)},
\end{equation*}
so that, when $k>2\pi d$,
\begin{equation*}
    F_\varphi(k)\le\frac{(d!)^2}{2\pi(k/(2\pi)-d)^{2d+1}}.
\end{equation*}
Let us write $R_{d,J}(t)$ for the remainder denoted in Section \ref{sec:gen} by $R_J(t)$, in the case where $\hat\varphi=\hat\varphi_d$. According to Remark \ref{rmk:extension_thm_exact}, we find, for $J>2t\mathcal L d+M$, 
\begin{equation*}
   \left|R_{d,J}(t)\right|\le \frac{{2}t\mathcal L}\pi\int_{\frac\pi{t\mathcal L}(J-M)}^{+\infty}\frac{(d!)^2}{2\pi(k/(2\pi)-d)^{2d+1}}\,dk=\frac{(d!)^2(2t\mathcal L)^{2d+1}}{2\pi d(J-M-2t\mathcal L d)^{2d}}.
\end{equation*}
Using the inequality $d!\le \sqrt{2\pi d}d^de^{-d+{1/12}}$ (see for instance Formula (3.9) in \cite{Ar64}), we find that
\begin{equation}\label{eq:eps_pre}
    \left|R_{d,J}(t)\right|\le\varepsilon(M,\rho,d,J),
\end{equation}
with  $\rho:=2t\mathcal L$ and
\begin{equation}\label{eq:eps}
    \varepsilon(\mu,\gamma,\alpha,\beta):=\frac{\alpha^{2\alpha}e^{-2\alpha+1/6}\gamma^{2\alpha+1}}{(\beta-\mu-\gamma\alpha)^{2\alpha}}.
\end{equation}
Let us note that we have replaced the discrete (integer) values $d$ and $J$ by the continuous parameters $\alpha$ and $\beta$.
\begin{theorem}\label{thm:opt_d}
For fixed $M$ and $\rho$ positive and $\overline\varepsilon\in(0,1)$, there exists $J^*$, the smallest among the integers $J$ such that $\varepsilon(M,\rho,d,J)\le\overline\varepsilon$ is satisfied for some $d$. There exists then $d^*$, the smallest positive integer $d$ such that $\varepsilon(M,\rho,d,J^*)\le\overline\varepsilon$. Furthermore, these integers are
given by the following formulas
\begin{equation}
    \label{eq:opt_J}
    J^*=\min\{\lceil\beta(\overline\varepsilon,M,\rho,d)\rceil\,:\,d\in\{\lfloor\alpha^*(\overline\varepsilon,\rho)\rfloor,\lceil\alpha^*(\overline\varepsilon,\rho)\rceil\}\}
\end{equation}
and
\begin{equation}
    \label{eq:opt_d}
    d^*=\min\{d\,:\,d\in\{1,\dots,\lceil\alpha^*(\overline\varepsilon,\rho)\rceil\}\mbox{ and }\lceil\beta(\overline\varepsilon,M,\rho,d)\rceil=J^*\},
\end{equation}
where
\begin{equation*}
        \beta(\overline{\varepsilon},M,\rho,d)=M+\rho d\left(1+\frac1e\left(\frac{e^{1/6}\rho}{\overline{\varepsilon}}\right)^{1/(2d)}\right)
    \end{equation*}
and
\begin{equation*}
        \alpha^*(\overline{\varepsilon},\rho)=\frac1{2(1+W(1))}\ln\left(\frac{e^{1/6}\rho}{\overline{\varepsilon}}\right).
    \end{equation*}
\end{theorem}
\begin{remark}\label{rem:wLambert} We denote by $W(1)$ the unique positive number satisfying $W(1)e^{W(1)}=1$. The notation $W(1)$ is standard and denotes the value of the Lambert W-function at $1$ (see \cite[Eq. 4.13.1]{NIST:DLMF}).
\end{remark}
\begin{remark}\label{rem:opt_d}
Existence of the integers $J^*$ and $d^*$ follows immediately from the fact that $\beta\mapsto\varepsilon(M,\rho,d,\beta)$ decreases to $0$ for any positive integer $M$ and $d$ and any $\rho>0$. Our actual goal is thus to prove Equations (\ref{eq:opt_J}) and (\ref{eq:opt_d}), which we use to compute them.
\end{remark}

\begin{proof}[Proof of Theorem \ref{thm:opt_d}] It follows from Equation (\ref{eq:eps}) that the unique $\beta\in(M+\rho\alpha,+\infty)$ which solves the equation $\varepsilon(M,\rho,\alpha,\beta)=\overline\varepsilon$ is 
\begin{equation}
\label{eq:beta}
    \beta(\overline{\varepsilon},M,\rho,\alpha)=M+\rho \alpha\left(1+\frac1e\left(\frac{e^{1/6}\rho}{\overline{\varepsilon}}\right)^{1/(2\alpha)}\right).
\end{equation}
This implies that, for any positive integer $d$, the smallest integer $J$ such that $\varepsilon(M,\rho,d,J)\le\overline\varepsilon$ is equal to $\lceil\beta(\overline\varepsilon,M,\rho,d)\rceil$.
Differentiating Equation (\ref{eq:beta}) with respect to $\alpha$, we find
\begin{equation*}
    \frac{d}{d\alpha}\beta(\overline\varepsilon,M,\rho,\alpha)=\rho\left(1+\frac1e\left(\frac{e^{1/6}\rho}{\overline{\varepsilon}}\right)^{1/(2\alpha)}\left(1-\frac1{2\alpha}\ln\left(\frac{e^{1/6}\rho}{\overline{\varepsilon}}\right)\right)\right).
\end{equation*}
Using the substitution 
\begin{equation*}
    s:=\frac1{2\alpha}\ln\left(\frac{e^{1/6}\rho}{\overline{\varepsilon}}\right)-1,
\end{equation*}
we see that this derivative vanishes if, and only if, $se^s=1$. This equation has a unique positive solution which is, by definition, $s=W(1)$. Therefore, the derivative vanishes only at $\alpha=\alpha^*(\overline\varepsilon,\rho)$, with
\begin{equation*}
    \alpha^*(\overline\varepsilon,\rho)=\frac1{2(1+W(1))}\ln\left(\frac{e^{1/6}\rho}{\overline{\varepsilon}}\right).
\end{equation*}
We can easily check that the derivative is negative below that value and positive above, so that $\alpha\mapsto\beta(\overline{\varepsilon},M,\rho,\alpha)$ is first decreasing and then increasing. The function $d\mapsto\beta(\overline{\varepsilon},M,\rho,d)$ takes its smallest value among all positive integers for the largest integer below $\alpha^*(\overline\varepsilon,\rho)$ or the smallest integer above. Equations (\ref{eq:opt_J}) and (\ref{eq:opt_d}) then follow easily.

\end{proof}

\section{Numerical examples}
\label{se:numerics}

We have numerically tested our method on some graphs for which we have computed the exact spectrum up to machine error precision and then added an error term uniformly distributed with prescribed variance.
In this section we give an overview of the implemented algorithm and show a few examples of its application.
{We recall the quantities that appear in the numerical analysis:
\begin{itemize}
    \item $t$ is the rescaling parameter of the test function;
    \item $d, (d^*)$ is the order of the test function $\varphi_d$, (respectively the minimal order according to equation (\ref{eq:opt_d}));
    \item $J, (J^*)$ is the number of the first smallest eigenfrequencies used to compute $S_J(t)$, (respectively the minimal number according to Equation (\ref{eq:opt_J}));
    \item $S_J(t)$ is the truncated sum defined in (\ref{eq:SR_def}) computed on the first $J$ eigenfrequencies of the exact spectrum where the test function is rescaled by $t$;
    \item $\widetilde{S}_J(t)$ is the same truncated sum as $S_J(t)$, but computed on a sample of the spectrum affected by a certain error term.
\end{itemize}
In view of the result of Proposition \ref{prp:exact} one might expect to obtain a good approximation of $\chi$ by computing the limit as $t \to \infty$ of the right hand side of (\ref{eq:chi_series}). Unfortunately this works out only if all eigenvalues are taken into account in the computation.
}

\subsection{Code overview}

The code consists of three distinct parts each of them developed to run on the software Mathematica \cite{MA11}.
\begin{enumerate}
    \item The first part is a GUI for building a graph with a drawing, drag-and-drop method and inserting the lengths of the edge.
    \item The second part contains an algorithm which computes the eigenvalues of the standard Laplacian of the graph in input.
    This is based on the von Below formula \cite{VB85} which holds for equilateral metric graphs, cf. \cite{Be13, Ku20}.
    This code shares parts with an other study in \cite{Ho20} on how this method can be extended to approximate eigenvalues of non equilateral metric graphs.
    Several samples of approximate spectra $\{\widetilde{k}_j\}_{j \in \mathbb{N}}$ are realized by summing a pseudo-random number uniformly distributed over a prescribed interval to the exact spectrum: $\widetilde{k}_j = k_j + e_j$,  $e_j \sim \mathcal{U}([-\delta,+\delta])$.
    \item The last part computes the parameters $d^*, J^*$, and Formula (\ref{eq:exact}), plots the maps $t,J \mapsto S_J(t)$ {and several instances of the the maps $t,J \mapsto \widetilde{S}_J(t)$ over different samples of spectra with error terms, and compares the error bound $\varepsilon(M,2t\mathcal{L},d,J) + 2 \delta J / t$ with the numerical errors $|\widetilde{S}_J(t)-\chi|$ over different values of $t$ and $J$}.
\end{enumerate}

\subsection{Comparison of the theoretical bounds with the experimental error}

\begin{figure}[ht]
\centering
    \includegraphics[width=.45\textwidth]{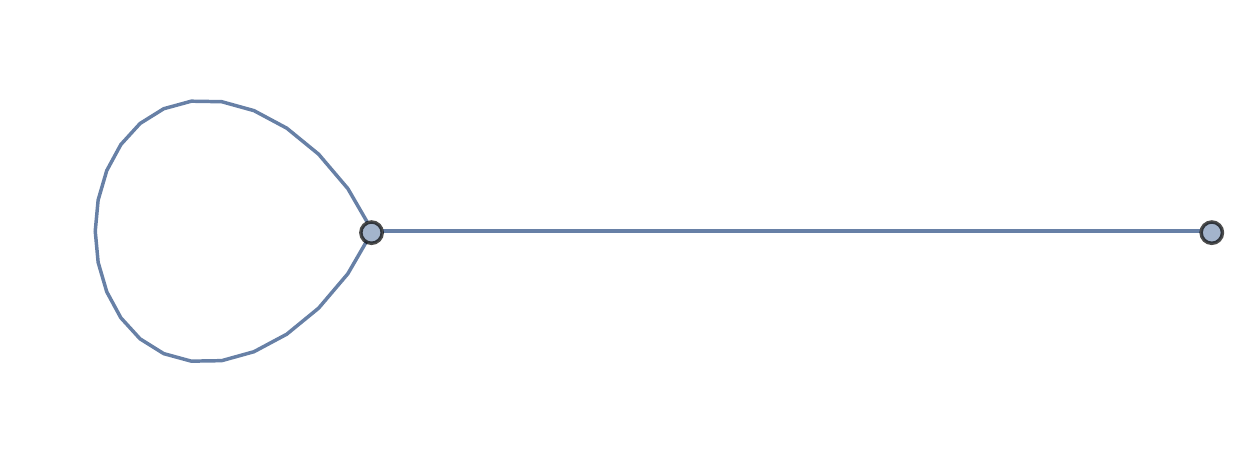}
    \includegraphics[width=.3\textwidth]{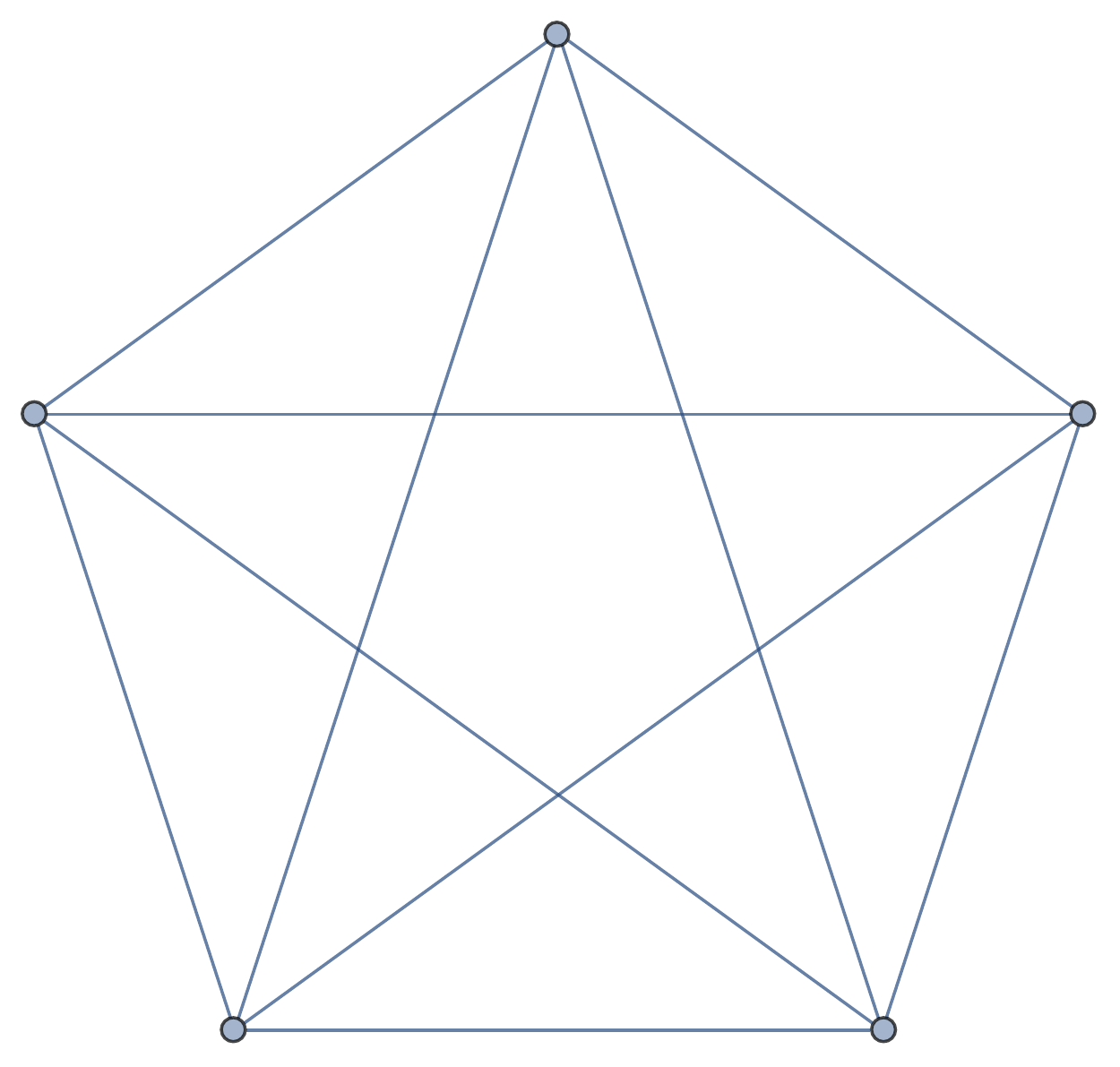}
    \caption{On the left the Lasso graph, $\chi = 0$, the loop has length $\ell_1 = 1$ and the pendant edge has length $\ell_2 = 5$, hence $\ell_{\mbox{\tiny min}} = 1$, $\rho = 12$ and $d^* = 1$; on the right the graph $K_5, \chi = -5$, each edge has assigned length one, hence $\ell_{\mbox{\tiny min}} = 2$, $\rho = 10$ and $d^* = 1$.
    The lengths of the edges are not to scale.}
    \label{fig:lasso_K5}
\end{figure}

Here we present some computations of {$S_J(t)$ and $\widetilde{S}_J(t)$ under different values of $J$ and $t$ compared with $\chi$} performed for the graphs Lasso and $K_5$ (see Figure \ref{fig:lasso_K5}). { We used the test function of optimal order $d^*$, equal to $1$ in both cases.}

\begin{figure}[ht]
\centering
    \includegraphics[width=.45\textwidth]{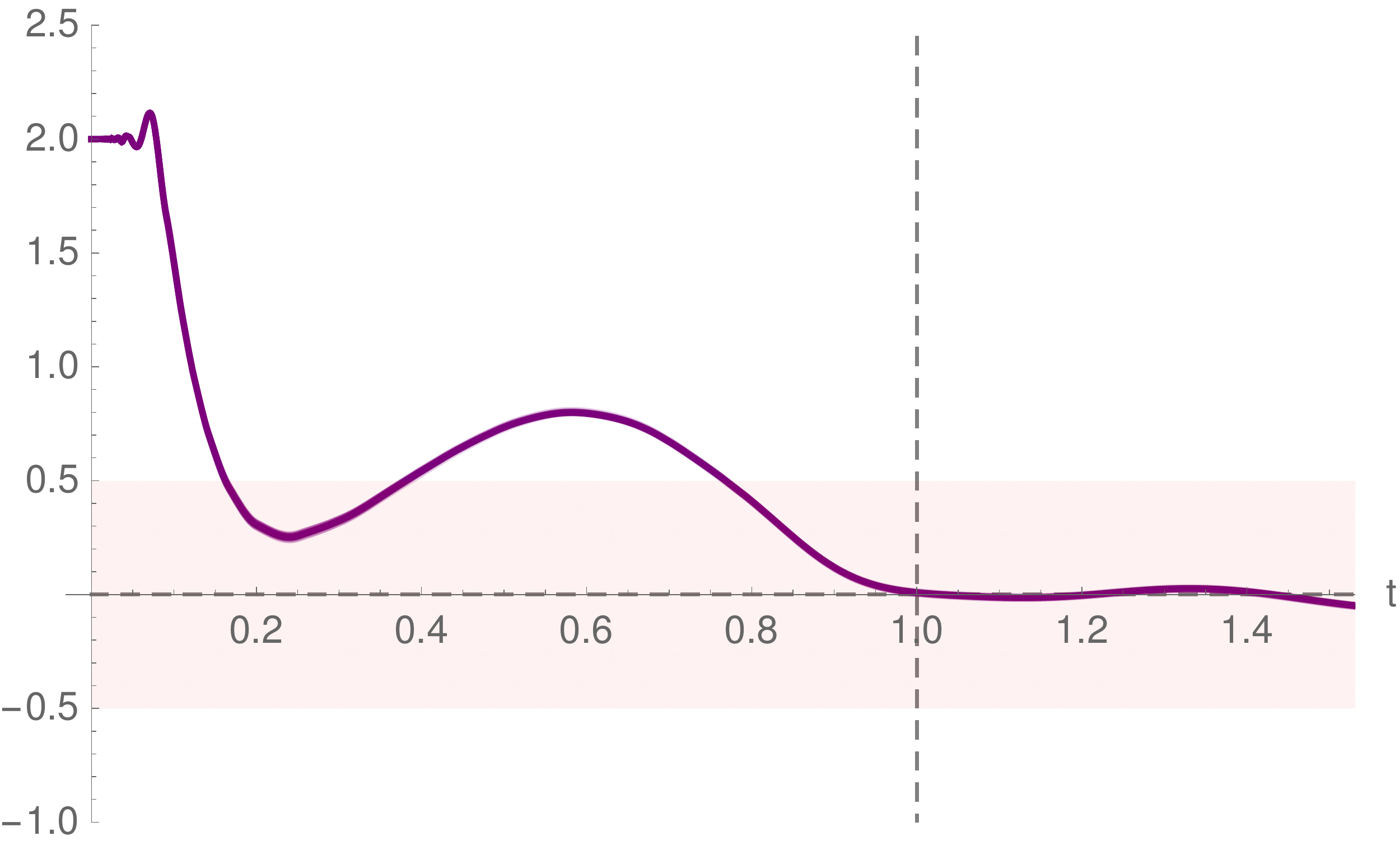}
    \includegraphics[width=.45\textwidth]{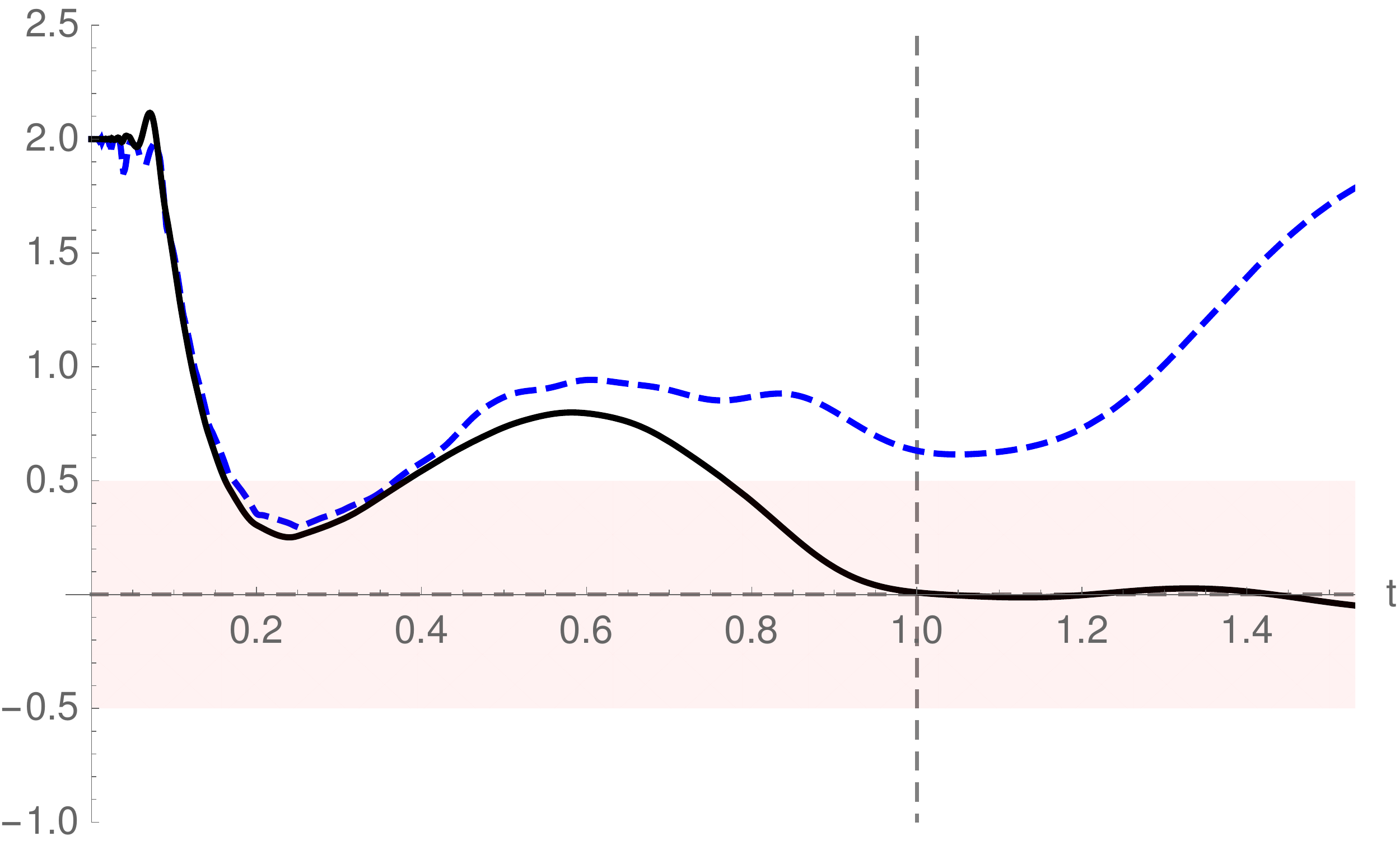}\\
    \includegraphics[width=.45\textwidth]{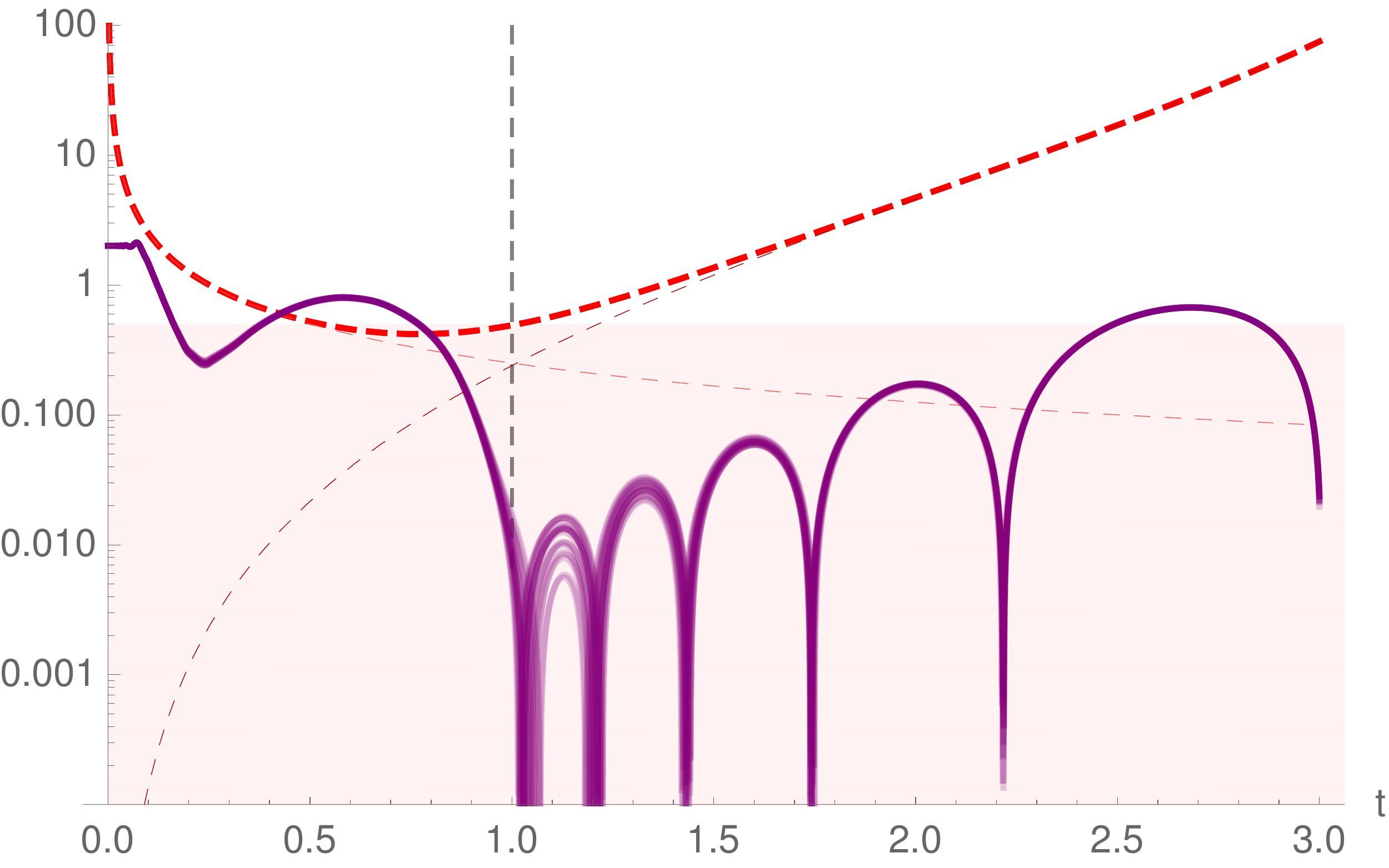}
    \includegraphics[width=.45\textwidth]{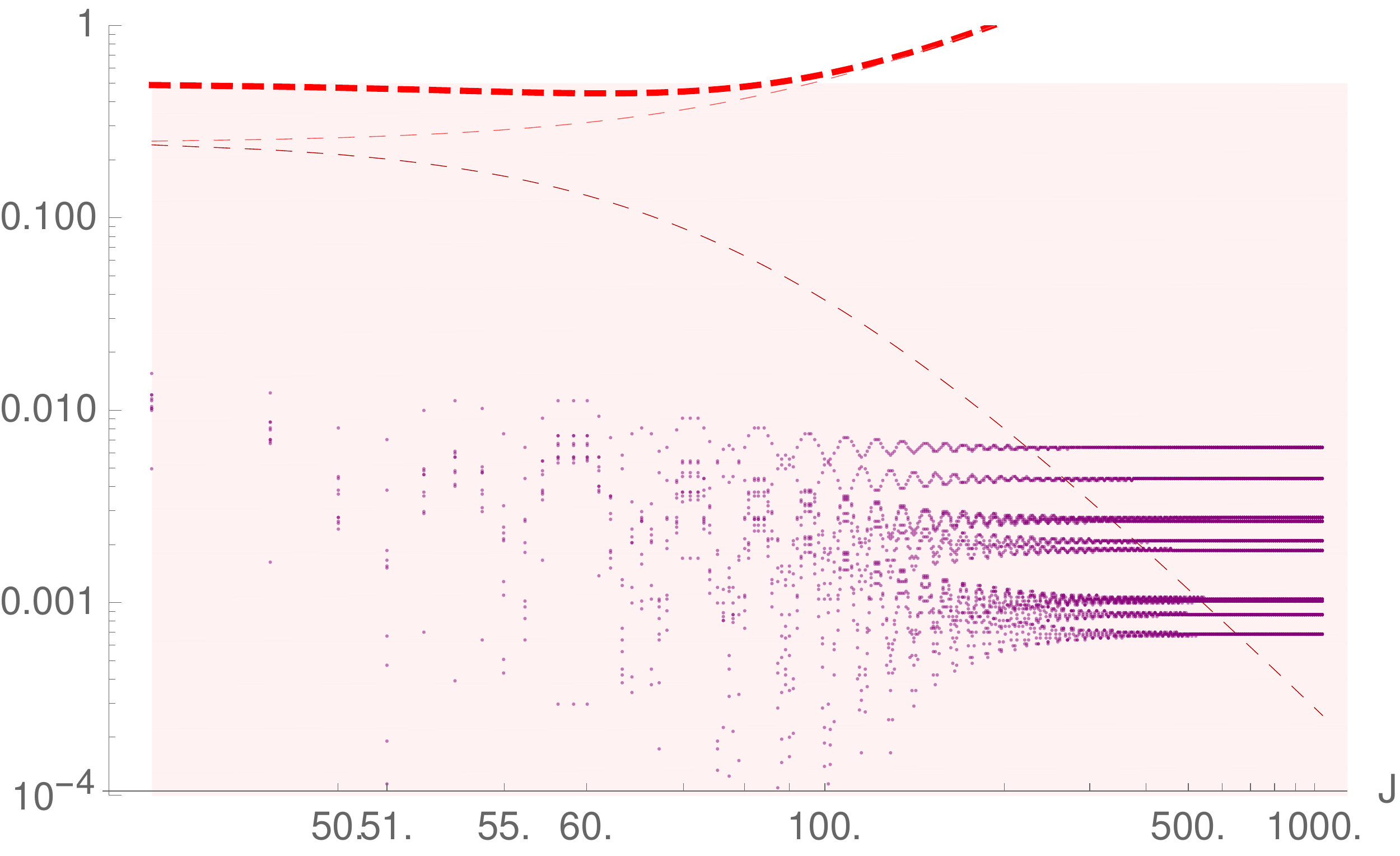}
    \caption{{Plots for the Lasso graph:
    In the first three plots $J = J^* = 48$ is fixed and the value $t = 1 / \ell_{\mbox{\tiny min}} = 1$ is shown with the dashed vertical line.
    In the bottom right plot $t = 1 / \ell_{\mbox{\tiny min}} = 1$ is fixed.
    The samples of approximate spectra are computed with $\delta \approx 2.6 \cdot 10^{-3}$.}}
    \label{fig:lasso_plots}
\end{figure}

Let us explain how the numerical results are displayed in Figures \ref{fig:lasso_plots} and \ref{fig:K5_plots}.
\begin{description}
        \item[Top left]: the function $t \mapsto S_J(t)$ (in black) is compared with
        
        { several instances of $t \mapsto \widetilde{S}_J(t)$, corresponding to different samples of the approximate spectrum.}
        The curves are not visibly distinguishable despite the different error terms.
        \item[Top right]: the function $t \mapsto S_J(t)$ is compared for two different test functions, $\varphi_d$ {(black solid curve)} and the triangular function $\psi$ {(blue dashed curve)}.
        \item[Bottom left]:
        
        { several instances of the error $t \mapsto |\chi - \widetilde{S}_J(t)|$ (purple solid curves) are compared with the theoretical bounds $t \mapsto \varepsilon(M,\rho,d,J) + 2\delta J / t$ (the red thick dashed curve, with the thin curves representing the two addends)}
        We recall that we have assumed $t \geq 1 / \ell_{\mbox{\tiny min}}$ when we computed the theoretical bounds.
        \item[Bottom right]:
        
        { several instances of the error  $J \mapsto |\chi - \widetilde{S}_J(t)|$ (purple dots) are compared with the theoretical bound $J \mapsto \varepsilon(M,\rho,d,J) + 2\delta J / t$ (the red thick dashed curve, with the thin curves representing the two addends)}.
\end{description}
\begin{figure}[ht]
\centering
    \includegraphics[width=.45\textwidth]{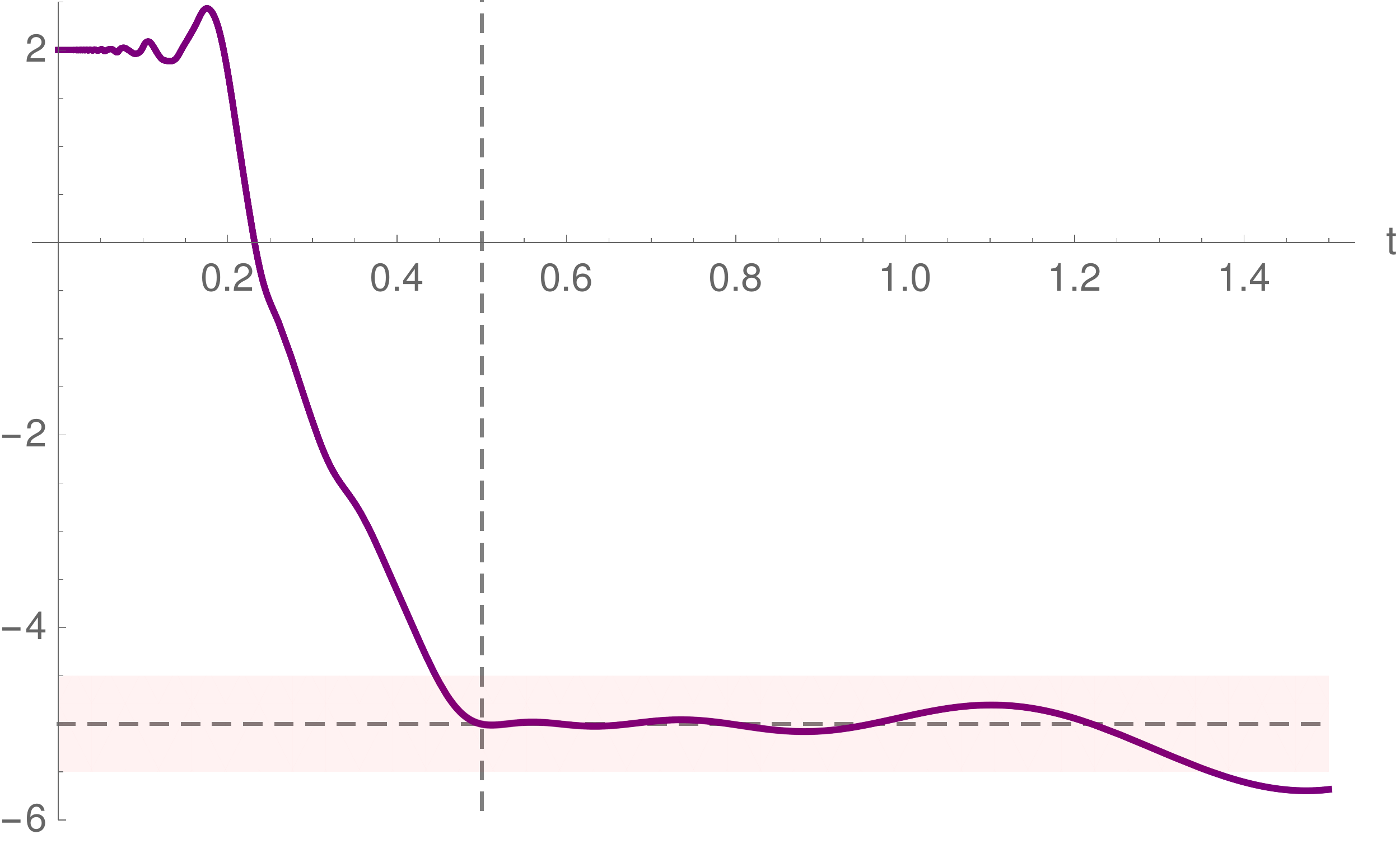}
    \includegraphics[width=.45\textwidth]{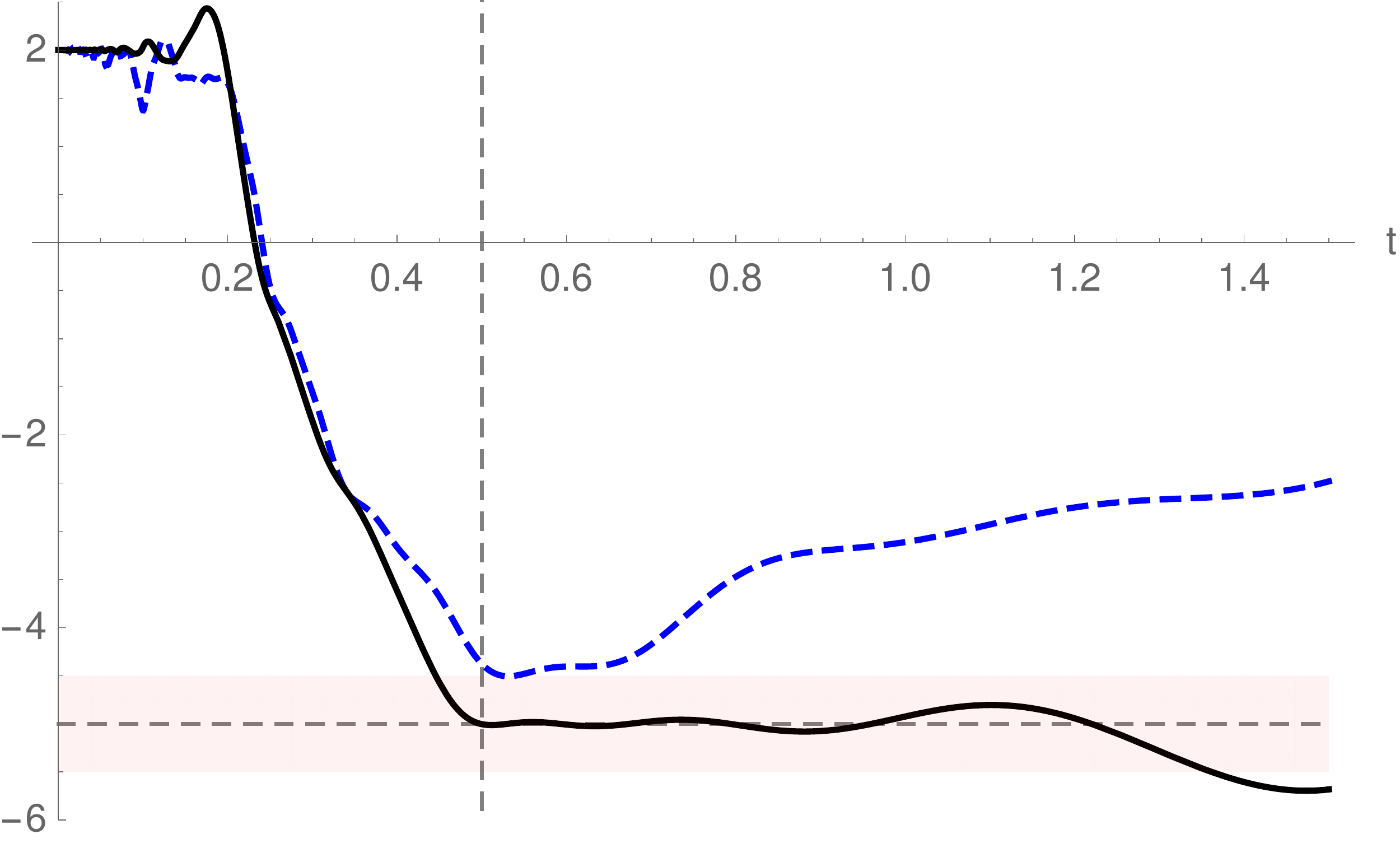}\\
    \includegraphics[width=.45\textwidth]{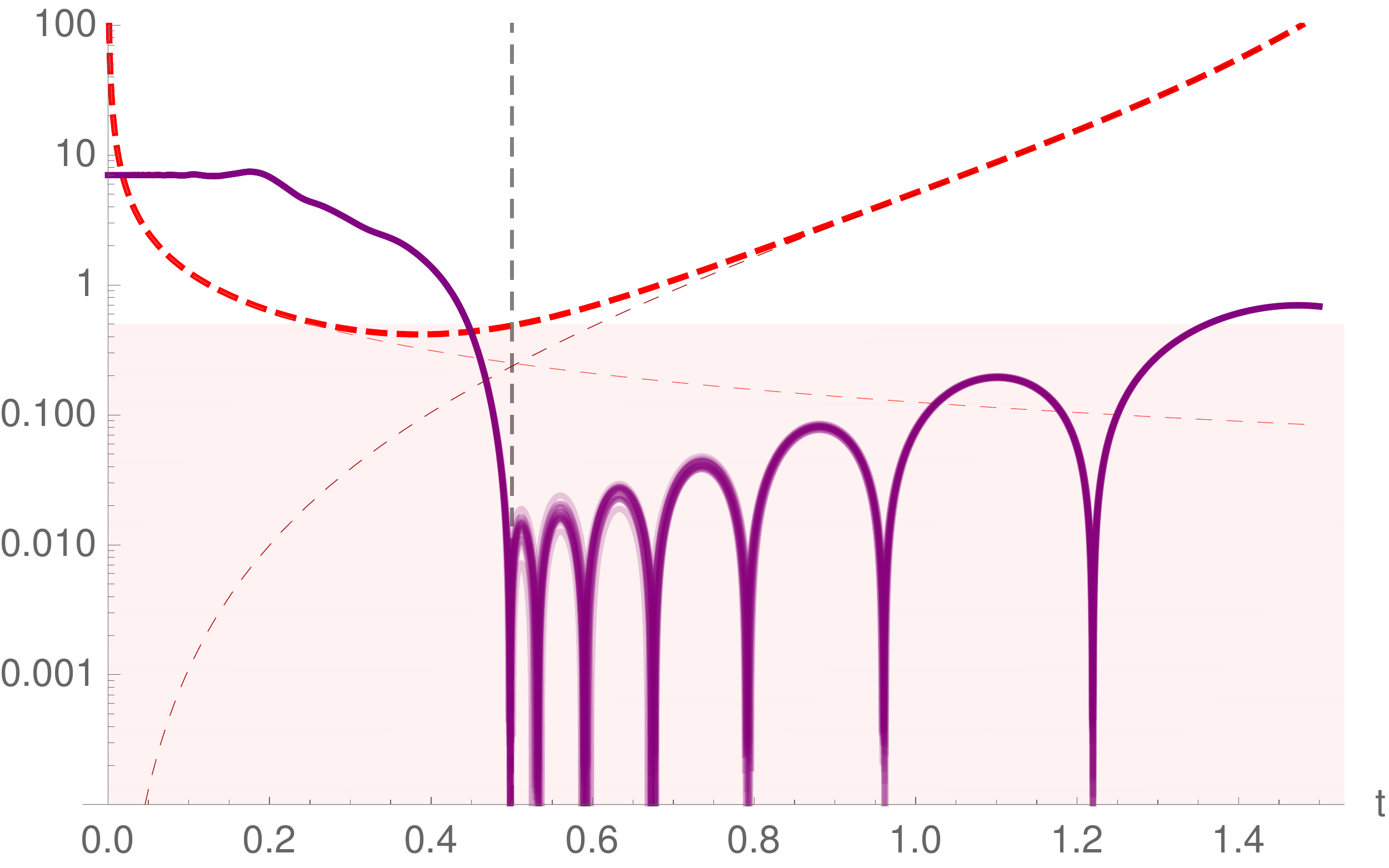}
    \includegraphics[width=.45\textwidth]{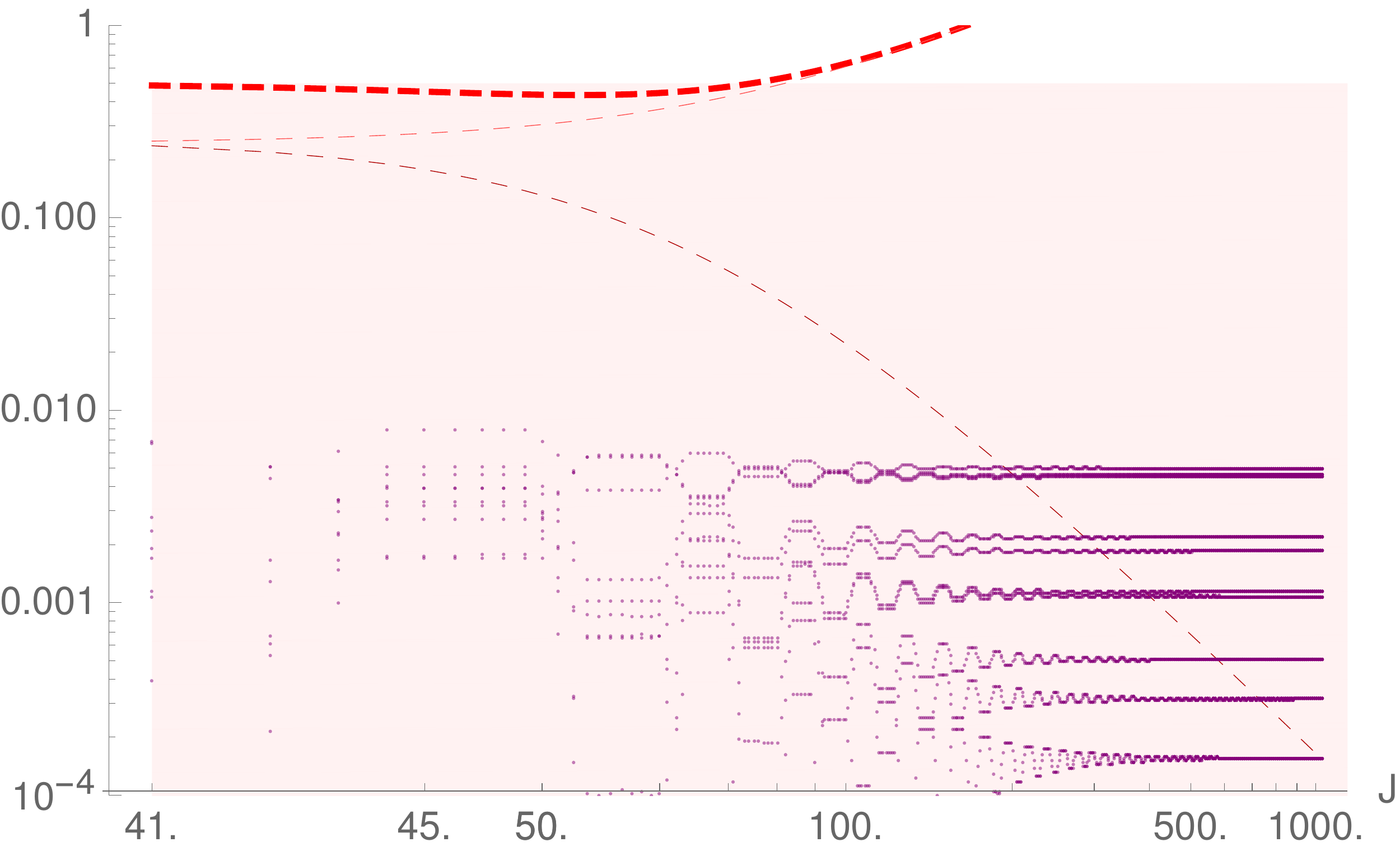}
    \caption{{Plots for the graph $K_5$:
    in the first three plots $J = J^* = 41$ is fixed and the value $t = {1 / \ell_{\mbox{\tiny min}}} = 0.5$ is shown with the dashed vertical line.
    In the bottom right plot $t = {1 / \ell_{\mbox{\tiny min}}} = 0.5$ is fixed.
    The samples of approximate spectra are computed with $\delta \approx 1.5 \cdot 10^{-3}$.}}
    \label{fig:K5_plots}
\end{figure}

\begin{observation}~
We now gather some observations which hold for both Figures \ref{fig:lasso_plots} and \ref{fig:K5_plots}.
\begin{enumerate}
    \item In the two plots at the top right we notice that the black curve approximates $\chi$ with higher precision than the blue curve, i.e. it stays within distance $1/2$ from $\chi$ for a certain interval in $t$.
    This shows that the method with $\varphi_d$ is significantly more efficient compared with $\psi$ in terms of number of eigenfrequencies sufficient in order to compute $\chi$.
    \item The two plots at the bottom show that the theoretical bound exceeds the numerical error by at least a factor $10^2$.
    This phenomenon can be explained by the several cancellations occurring inside the tail $R_{d,J}(t)$ which we did not take into account in (\ref{eq:eps_pre}).
    \item {In the bottom left plots it can be observed that both the theoretical bound and the numerical error grow with respect to the parameter $t$, indicating that it is important for the effectiveness of the method to have a good lower bound of $\ell_{\mbox{\tiny min}}$.}
    \item In the bottom right plot we observe that for $J$ growing the numerical error stabilizes at small positive values different for different samples.
    This suggests that the precision does not improve by taking $J \gg J^*$ and that the error mainly depends on the contribution from the first terms in the summation $\widetilde{S}_J(t)$.
    \item The theoretical bound shows to be even worse with respect to the numerical error for $J \gg J^*$, this can be explained by the decreasing magnitude of the Fourier coefficients of the test function $\varphi$ inside $|\widetilde{S}_J(t) - S_J(t)|$, which we did not take into account in  (\ref{eq:error:start}--\ref{eq:error:end}).
\end{enumerate}
\end{observation}

\subsection{Additional examples}

We compare the numerical computation of $\chi$ for the following three equilateral graphs $K_5$, $K_5$ with one edge made pendant, i.e. disconnected at one end-point, and $K_{3,3}$.
These graphs have Euler characteristic respectively $\chi = -5, -4$ and $-3$, they have all three been constructed equilateral with each edge of length one.
The first two graphs are chosen to have very close spectra. We can numerically observe that only $1$ in $5$ of their eigenfrequencies differ significantly (figure \ref{fig:comp_plots}).

\begin{figure}[ht]
\centering
    \includegraphics[width=.45\textwidth]{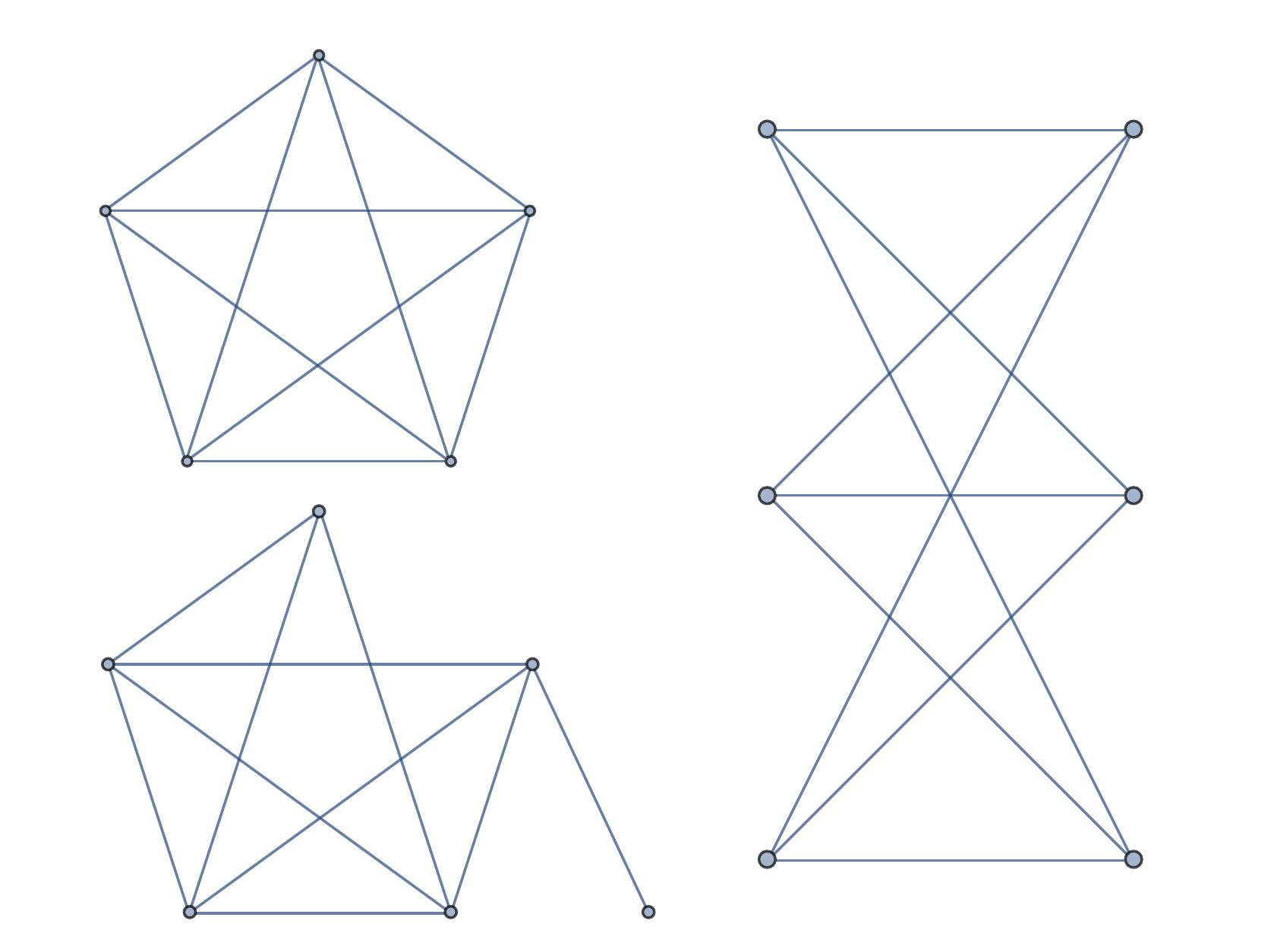}
    \includegraphics[width=.45\textwidth]{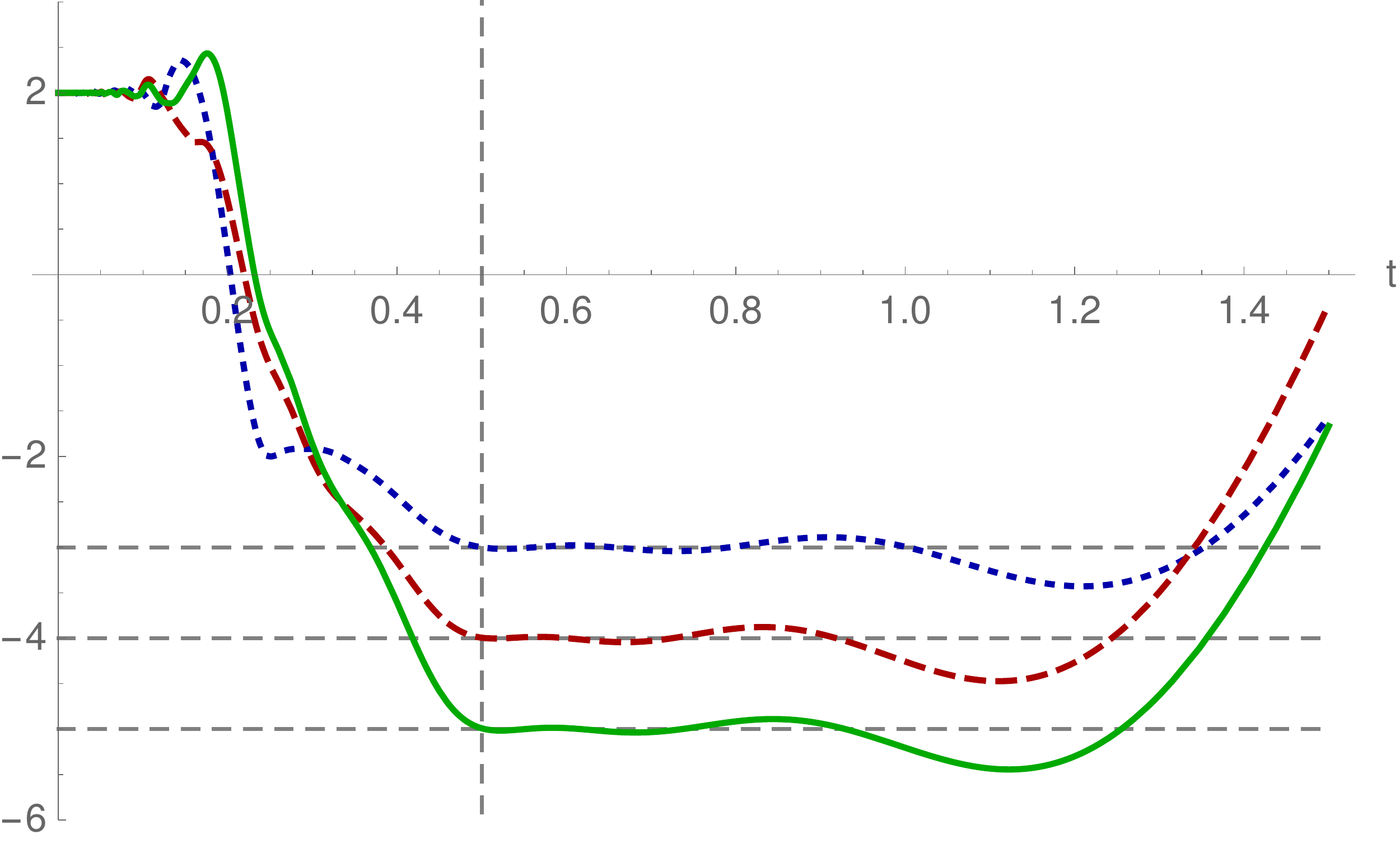}
    \caption{{For each graph in figure we plot the functions $t \mapsto S_{J}(t)$ with $d=1$ and $J = 30$: $K_{3,3}$ (blue dotted curve), $K_5$ with one endpoint disconnected (red dashed curve) and $K_5$ (green solid curve).}}
    \label{fig:comp_plots}
\end{figure}

\subsection{Influence of the {order} of the test function}

From Theorem \ref{thm:opt_d} we can notice that the values of $d^*$ and $J^*-M = \lceil \beta(\overline{\varepsilon},M,\rho,d) - M \rceil$ depend only on the parameters $\rho$ and $\overline{\varepsilon}$.
Following the idea described in Section \ref{ss:approxspectrum}, we fix $\overline{\varepsilon} = 1/4$ and, for increasing values of $\rho$, we compute the values of $d^*$ and bounds of $J^*-M$. The results are shown in the following table.
{
\begin{center}
\begin{tabular}{r c r | c | r c r}
& $\rho$ & & $d^*$ & & $J^* - M$ & \\
 \hline
 $2$ & $\leq \rho \leq $ & $15.6$    & $1$       & $5$ & $\leq J^* - M \leq$ & $65$  \\ 
 $15.6$ & $\leq \rho \leq $ & $16.5$    & $1$ or $2$& $65$ & $\leq J^* - M \leq$ & $70$  \\
 $16.5$ & $\leq \rho \leq $ & $421$     & $2$       & $70$ & $\leq J^* - M \leq$ & $2912$\\
 $421$ & $\leq \rho \leq $ & $423$     & $2$ or $3$& $2912$ & $\leq J^* - M \leq$ & $2925$\\
 $423$ & $\leq \rho \leq $ & $10^4$    & $3$       & $2925$ & $\leq J^* - M \leq$ & $10^5$
\end{tabular}
\end{center}}

Cases for which $d^*>1$ occur for $\rho > 16.5$ i.e. $\mathcal{L} >  8.3 \cdot \ell_{\mbox{\tiny min}}$. This depends on two factors: the distribution of the length of the edges and the number of edges.
We can identify two extreme cases.
\begin{itemize}
    \item The distribution of the edge length has large variance, i.e. the length of the shortest edge is very small compared to the total length.
    For instance, a Lasso graph with the loop at least $8$ times shorter than the pendant edge.
    \item The distribution of the edge length has null, or very small variance, but the number of edges is relatively large, at least $17$ if the graph does not have loops.
    For instance, the equilateral $K_7$.
\end{itemize}

\appendix
\section{Proof of Proposition \ref{prp:exact}}\label{appTrace}

A priori, we only know that Formula (\ref{eq:series}) holds for test functions $\varphi$ in the Schwartz class. In order to extend it to functions in the set $\mathcal T$, we use a standard argument of regularization by convolution. Let us give more details.   

We choose a function $\theta:\mathbb R\to \mathbb R$ of class $C^\infty$, non-negative, whose support is in $[-1,1]$. We also require that $\theta$ is even and that $\int_{-\infty}^{+\infty}\theta(\ell)\,d\ell=1$. It follows that $\hat\theta$ is real-valued and even, with $\hat\theta(0)=1$ and $|\hat\theta(k)|\le 1$ for all $k\in \mathbb R$. For any { positive integer $n$}, we define {$\theta_n(\ell):=n\theta(n\ell)$}. We have {$\hat\theta_n(k)=\hat\theta\left(\frac{k}{n}\right)$} for all $k\in \mathbb R$.

Let us fix {$\varphi\in\mathcal T$ and a positive integer $n$}. The function { $\varphi_{n}:=\theta_n*\varphi$} is of class $C^\infty$ with support in {$\left[-\frac1n,L+\frac1n\right]$ and satisfies $\int_{-\infty}^{+\infty}\varphi_n(\ell)\,d\ell=1 $.} Formula (\ref{eq:series}) applies and gives us
{
\begin{align}
 \chi+2\mathcal L \varphi_{n}(0)+\sum_{p\in\mathcal P}\ell(\mbox{prim}(p))S_V(p)\left(\varphi_{n}(-\ell(p))+\varphi_{n}(\ell(p))\right)\nonumber\\
=2 \hat {\varphi}(0) +2\sum_{k_j>0}{\hat\theta\left(\frac{ k_j}n\right)}\Re\left(\hat {\varphi}(k_j)\right),\label{eq:series_app}
\end{align}
}
where the series in the left-hand side reduces to a finite sum.

For all $\ell\in \mathbb R$, { $\varphi_{n}(\ell)\to \varphi(\ell)$ as $n\to \infty$}, so the left-hand side of Equation (\ref{eq:series_app}) converges to 
\begin{equation*}
\chi+\sum_{p\in\mathcal P}\ell(\mbox{prim}(p))S_V(p)\left(\varphi(-\ell(p))+\varphi(\ell(p))\right).
\end{equation*}
Condition (iv) in the definition of the set $\mathcal T$ implies that the series 
\begin{equation*}
    \sum_{k_j>0}\Re\left(\hat {\varphi}(k_j)\right)
\end{equation*}
is absolutely convergent. On the other hand, 
\begin{equation*}
\left|\hat\theta\left(\frac{ k_j}{n}\right)\Re\left(\hat {\varphi}(k_j)\right)\right|\le \left|\Re\left(\hat {\varphi}(k_j)\right)\right|,
\end{equation*}
while ${\hat\theta\left(\frac{ k_j}{n}\right)}\to \hat\theta(0)=1$ as $n\to\infty$. It follows that
\begin{equation*}
\sum_{k_j>0}\hat\theta\left(\frac{ k_j}n\right)\Re\left(\hat {\varphi}(k_j)\right)\to \sum_{k_j>0}\Re\left(\hat {\varphi}(k_j)\right)
\end{equation*}
as $n\to\infty$, using for instance the Dominated Convergence Theorem for the counting measure associated with the sequence $(k_j)_{j\ge1}$. Taking the limit $n\to\infty$ in Equation (\ref{eq:series_app}) therefore yields the desired result.

{
\section{Fourier transform of the test functions}\label{appFT}
}

Let us recall that we are considering the family of test functions
\begin{equation*}
    \varphi_d(\ell)=C_d\left(1-\cos(2\pi\ell)\right)^d\mathds{1}_{[0,1]}(\ell),
\end{equation*}
with $d$ a positive integer. To apply Proposition \ref{prp:exact}, we need to compute the Fourier transform of $\varphi_d$, and also to determine the normalisation constant $C_d$. Let us set, for $k\in \mathbb R$,
\begin{equation*}
    I_d(k):=\int_0^1 \left(1-\cos(2\pi\ell)\right)^de^{ik\ell}\,d\ell.
\end{equation*}
Standard trigonometric identities give us
\begin{equation*}
    I_d(k):=2^d\int_0^1\sin(\pi\ell)^{2d}e^{ik\ell}\,d\ell=(-1)^d2^{-d}\int_0^1(e^{i\pi\ell}-e^{-i\pi\ell})^{2d}e^{ik\ell}\,d\ell.
\end{equation*}
We can evaluate quickly the integral on the right using complex integration, and assuming $k>2\pi d$. Let us denote by $\gamma$ the closed path in the complex plane $\mathbb C$ obtained by adding the path $\gamma_0$, consisting of the segment $[-1,1]$ traversed in the positive direction, and $\gamma_1$, the semi-circle centered at $0$ traversed in the anti-clockwise direction from $1$ to $-1$. Then $I_d(k)=(-1)^d2^{-d}/(i\pi)\int_{\gamma_1}f(z)\,dz$, with
\begin{equation*}
    f(z)=\frac1z\left(z-\frac1z\right)^{2d}z^{k/\pi}=(z^2-1)^{2d}z^{k/\pi-2d-1}.
\end{equation*}
Since $\int_\gamma f(z)\,dz=0$, we obtain 
\begin{equation*}
    I_d(k)=-(-1)^d2^{-d}/(i\pi)\int_{\gamma_0}f(z)\,dz.
\end{equation*}
We find
\begin{equation*}
    I_d(k)=\frac{(-1)^{d+1}}{2^{d}i\pi}\left(\int_0^1(t^2-1)^{2d}t^{ k/\pi-2d-1}\,dt-e^{ik}\int_0^1(t^2-1)^{2d}t^{ k/\pi-2d-1}\,dt\right),
\end{equation*}
the two terms on the right-hand side being the integral on the positive and negative part of $\gamma_0$, respectively. Using the change of variable $t=\sqrt{s}$, we transform the expression into
\begin{align*}
     I_d(k)&=\frac{(-1)^{d+1}}{2^{d+1}i\pi}(1-e^{ik})\int_0^1(s-1)^{2d}{s}^{k/(2\pi)-d-1}\,{ds}\\
     &=\frac{(-1)^d}{2^d\pi}e^{ik/2}\sin(k/2)B(2d+1,k/(2\pi)-d),
\end{align*}
where $B$ denotes the Euler beta function. We have
\begin{equation*}
    B(2d+1,k/(2\pi)-d)=\frac{\Gamma(2d+1)\Gamma(k/(2\pi)-d)}{\Gamma(k/(2\pi)+d+1)}=\frac{(2d)!}{\Pi_{j=-d}^d(k/(2\pi)+j)},
\end{equation*}
where the last equality results from using, $2d+1$ times, the identity $\Gamma(x+1)=x\Gamma(x)$. We obtain
\begin{equation*}
     I_d(k)=\frac{(-1)^d(2d)!}{2^d\pi\Pi_{j=-d}^d(k/(2\pi)+j)}e^{ik/2}\sin(k/2).
\end{equation*}
Both sides of the equation define a function of $k$ which is holomorphic in $\mathbb C$. The equality, which we have proved { for $k >2\pi d$}, therefore holds for $k\in\mathbb C$ by the principle of isolated zeros. In particular, we find
\begin{equation*}
    C_d^{-1}=I_d(0)=\frac{(2d)!}{2^d(d!)^2},
\end{equation*}
and finally
\begin{equation*}
    \hat \varphi_d(k)=\frac{(-1)^d(d!)^2}{\pi\Pi_{j=-d}^d(k/(2\pi)+j)}e^{ik/2}\sin(k/2).
\end{equation*}

\section*{Acknowledgement}
The authors are grateful to Pavel Kurasov for suggesting to study the problem, providing valuable ideas, and allowing them to consult the preprints of \cite{Ku20,Mi20}. C.L. was partially supported by the Swedish Research Council (Grant D0497301).
The authors wish to thank the two anonymous referees for the corrections and useful remarks.
\bibliography{references}{}
\bibliographystyle{unsrt}

\end{document}